\documentclass{amsart}
\usepackage{amsfonts,amssymb,latexsym,amsthm,amsmath}
\usepackage{euscript}
\newtheorem{theorem}{Theorem}

\newtheorem{lemma}[theorem]{Lemma}

\newtheorem{proposition} [theorem]{Proposition}

\newtheorem{corollary}[theorem]{Corollary}

\def\S{\Sigma}
\def\Z{\mathbb{Z}}

\usepackage{tikz}

\usetikzlibrary{matrix}
\usepackage[arrow,matrix,curve]{xy}\SilentMatrices
\def\xyma{\xymatrix@M.7em}

\begin{document}
\title{Homotopy theory and generalized dimension subgroups}
\author{Sergei O. Ivanov}
\address{Chebyshev Laboratory, St. Petersburg State University, 14th Line, 29b,
Saint Petersburg, 199178 Russia} \email{ivanov.s.o.1986@gmail.com}

\author{Roman Mikhailov}
\address{Chebyshev Laboratory, St. Petersburg State University, 14th Line, 29b,
Saint Petersburg, 199178 Russia and St. Petersburg Department of
Steklov Mathematical Institute} \email{rmikhailov@mail.ru}
\author{Jie Wu }
\address{Department of Mathematics, National University of Singapore, 10 Lower Kent Ridge Road, Singapore 119076} \email{matwuj@nus.edu.sg}
\urladdr{www.math.nus.edu.sg/\~{}matwujie}

\thanks{The main result (Theorem~\ref{main}) is supported by Russian Scientific Foundation, grant N 14-21-00035. The last author is also partially supported by the Singapore Ministry of Education research grant (AcRF Tier 1 WBS No. R-146-000-190-112)
and a grant (No. 11329101) of NSFC of China.}

\begin{abstract}
Let $G$ be a group and $R,S,T$ its normal subgroups. There is a
natural extension of the concept of commutator subgroup for the
case of three subgroups $\|R,S,T\|$ as well as the natural
extension of the symmetric product $\|\bf r,\bf s,\bf t\|$ for
corresponding ideals $\bf r,\bf s, \bf t$
 in the integral group ring $\mathbb Z[G]$. In this
paper, it is shown that the generalized dimension subgroup $G\cap
(1+\|\bf r,\bf s,\bf t\|)$ has exponent 2 modulo $\|R,S,T\|.$ The
proof essentially uses homotopy theory. The considered generalized
dimension quotient of exponent 2 is identified with a subgroup of
the kernel of the Hurewicz homomorphism for the loop space over a
homotopy colimit of classifying spaces.
\end{abstract}

 \maketitle

\section{Introduction}

Let $G$ be a group and $\mathbb Z[G]$ its integral group ring.
Every two-sided ideal $\mathfrak a$ in the integral group ring
$\mathbb Z[G] $ of a group $G$ determines a normal subgroup $$D(G,
{\bf a}):=G \cap (1 + {\bf a})$$ of $G$. Such subgroups are called
{\it generalized dimension subgroups}. The identification of
generalized dimension subgroups is a fundamental problem in the
theory of group rings. In general, given an ideal ${\bf a}$, the
identification of $D(G,{\bf a})$ is very difficult, for a survey
on the problems in this area see \cite{Gupta},
\cite{MikhailovPassi}.

The idea that {\it the generalized dimension subgroups are related
to the kernels of Hurewicz homomorphisms of certain spaces} was
discussed in \cite{MikhailovPassi}, \cite{MPW}, however, in the
cited sources, all application of homotopical methods to the
problems of group rings were related to very special cases. In
this paper, we apply homotopy theory for a purely
group-theoretical result of a more general type, namely to the
description of the exponent of generalized dimension quotient
constructed for a triple of normal subgroups in any group $G$.

Let $G$ be a group and $R,S$ its normal subgroups. Denote ${\bf
r}=(R-1)\mathbb Z[G],\ {\bf s}=(S-1)\mathbb Z[G]$. It is proved in
\cite{BD} that \begin{equation}\label{bd} D(G,{\bf r}{\bf s}+{\bf
s}{\bf r})=[R,S].
\end{equation}
The following question arises naturally: how one can generalize
the result (\ref{bd}) to the case of three and more normal
subgroups of $G$. Our main result is the following.

\begin{theorem}\label{main}
Let $G$ be a group and $R,S,T$ its normal subgroups. Denote $${\bf
r}=(R-1)\mathbb Z[G],\ {\bf s}=(S-1)\mathbb Z[G],\ {\bf
t}=(T-1)\mathbb Z[G]$$ and
\begin{align*}
& \|R,S,T\|:=[R,S\cap T][S,R\cap T][T,R\cap S]\\ & \|\bf r, \bf
s,\bf t\|:=\bf r(\bf s\cap \bf t)+(\bf s\cap \bf t)\bf r+\bf s(\bf
r\cap \bf t)+(\bf r\cap \bf t)\bf s+\bf t(\bf r\cap \bf s)+(\bf
r\cap \bf s)\bf t.
\end{align*}
Then, for every $g\in D(G, \|{\bf r},{\bf s},{\bf t}\|)$, $g^2\in
\|R,S,T\|$, i.e. the generalized dimension quotient
\begin{equation}\label{dimension}
\frac{D(G, \|{\bf r},{\bf s},{\bf t}\|)}{\|R,S,T\|}
\end{equation}
is a $\mathbb Z/2$-vector space.
\end{theorem}

The proof of theorem \ref{main} consists of the following steps.
First we show that there exists a space $X$ such that there is a
commutative diagram
\begin{equation}\label{jq}
\xyma{\frac{R\cap S\cap T}{\|R,S,T\|}\ar@{->}[r]\ar@{=}[d] &
\frac{\bf
r\cap \bf s\cap \bf r}{\|\bf r,\bf s,\bf t\|}\ar@{->}[d]\\
\pi_2(\Omega X)\ar@{->}[r]^{h_2\Omega}& H_2(\Omega X)}
\end{equation}
where the lower horizontal map is the Hurewicz homomorphism.
Secondly, we show that, for any space $X$, the kernel of the
Hurewicz homomorphism
$$
\Omega h_2: \pi_2(\Omega X)\to H_2(\Omega X)
$$
is a 2-torsion subgroup of $\pi_2(\Omega X)=\pi_3(X)$.

There are examples of groups with triples of subgroups such that
the generalized dimension quotient \ref{dimension} is non-trivial.
Let $F=F(a,b,c)$ be a free group with basis $\{a,b,c\}$. Consider
the following normal subgroups of $F$:
$$
R=\langle a^2,c\rangle^F,\ S=\langle a,bc^{-1}\rangle^F,\
T=\langle a,b\rangle^F.
$$
Then, there exists the following natural commutative diagram
$$
\xyma{\mathbb Z/2\ar@{>->}[r] \ar@{=}[d] & \frac{R\cap S\cap
T}{\|R,S,T\|}\ar@{->}[r]\ar@{=}[d] & \frac{\bf r\cap \bf s\cap \bf
r}{\|\bf r,\bf s,\bf t\|}\ar@{=}[d]\\ \mathbb Z/2\ar@{>->}[r] &
\pi_2(\Omega \Sigma \mathbb RP^2)\ar@{->}[r]^{h_2\Omega}&
H_2(\Omega \Sigma \mathbb RP^2)}
$$
This example and discussion of a generalization of the considered
construction to the case of $>3$ normal subgroups is given in
section \ref{lastsection}.

Another application of homotopic methods is the following
identification of the generalized dimension subgroup (see theorem
\ref{modg}):
\begin{equation}\label{bd11}
D(G,{\bf rs+sr}+({\bf r}\cap {\bf s}){\bf t}+{\bf t}({\bf r}\cap
{\bf s}))=[R,S][R\cap S,T].
\end{equation}
This generalizes (\ref{bd}), indeed, (\ref{bd}) is equivalent to
(\ref{bd11}) for $T=1$.

The space $X$ from (\ref{jq}) is the homotopy colimit of the cubic
diagram of eight classifying spaces $BG, B(G/R), B(G/S), B(G/T),
B(G/RS), B(G/RT), B(G/ST)$. The left vertical isomorphism in
(\ref{jq}) is proved in \cite{EM}. In section \ref{section3} we
develop the theory of cubes of fibrations in the category of
simplicial non-unital rings and correspondence between $n$-cubes
of fibrations with crossed $n$-cubes of rings. We obtain
ring-theoretical analogs of the result from \cite{ES}. Note that,
in this paper, we do not consider the properties of universality
of crossed $n$-cubes of rings. The universality property is needed
for an explicit description of the homology groups $H_*(\Omega X)$
of homotopy colimits $X$ of classifying spaces (see the proof of
theorem 1 in \cite{EM}). For the reason of this paper, namely, for
an analysis of generalized dimension subgroups, only crossed
properties of the diagrams of rings are enough and these
properties are given in section \ref{section3}.

\section{Hurewicz homomorphism}

\subsection{Two lemmas about squares of abelian groups}
First we  state two lemmas about squares of abelian groups. These
lemmas are advanced versions of well known statements (see
\cite[Part 1, 6.2.6]{Faith}). We give them without a prove because
it is standard. Authors learned these lemmas from non-formal
discussions with Alexander Generalov.

Consider a square of abelian groups  $\mathcal S$ with induced
homomorphisms on kernels and cokernels:
$$
\xymatrix@-5mm{&& {\rm Ker}(g)\ar[dd]\ar[rr]^{\tilde f} && {\rm Ker}(g')\ar[dd]&& \\ &&&&&& \\
{\rm Ker}(f)\ar[rr]\ar[dd]^{\tilde g} && A \ar[rr]^f \ar[dd]^{g} && B\ar[dd]^{g'}\ar[rr] && {\rm Coker}(f)\ar[dd]^{\tilde g'}\\
&& &\mathcal S & &&\\
{\rm Ker}(f')\ar[rr]&& C\ar[dd]\ar[rr]^{f'} && D\ar[dd]\ar[rr]
&& {\rm Coker}(f')\\ &&&&&& \\
&& {\rm Coker}(g)\ar[rr]^{\tilde f'} && {\rm Coker}(g'), && }
$$
and maps $A\to B\oplus C \to D$ given by $a\mapsto (f(a),-g(a))$
and $(b,c)\mapsto g'(b)+f'(c).$
\begin{lemma}\label{Lemma_pushout}
The following statements are equivalent.
\begin{enumerate}
\item $\mathcal S$ is a pushout square. \item The sequence $A\to
B\oplus C \to D \to 0$ is exact. \item  $\tilde g'$ is an
isomorphism and $\tilde g$ is an epimorphism. \item $\tilde f'$ is
an isomorphism and $\tilde f$ is an epimorphism.
\end{enumerate}
\end{lemma}
\begin{lemma}\label{Lemma_pullback}
The following statements are equivalent.
\begin{enumerate}
\item $\mathcal S$ is a pullback square. \item The sequence $0\to
A\to B\oplus C \to D$ is exact. \item  $\tilde g$ is an
isomorphism and $\tilde g'$ is a  monomorphism. \item $\tilde f$
is an isomorphism and $\tilde f'$ is a monomorphism.
\end{enumerate}
\end{lemma}

\subsection{Whitehead quadratic functor}
For an abelian group $A$, the Whitehead group $\Gamma(A)$ is
generated by symbols $\gamma(a),\ a\in A$ with the following
relations
\begin{align*}
& \gamma(0)=0,\\
& \gamma(-a)=\gamma(a),\ a\in A\\
&
\gamma(a+b+c)-\gamma(a+b)-\gamma(a+c)-\gamma(b+c)+\gamma(a)+\gamma(b)+\gamma(c)=0,\
a,b,c\in A.
\end{align*}
The correspondence $A\mapsto \Gamma(A)$ defines a quadratic
functor in the category of abelian groups called the {\it
Whitehead quadratic functor}. It has the following simple
properties
\begin{align*} & \Gamma(\mathbb Z/n)=\mathbb Z/(2n,n^2)\\
& A\otimes B={\rm Ker}\{\Gamma(A\oplus B)\to \Gamma(A)\oplus
\Gamma(B)\}.
\end{align*}
There is a natural transformation of functors $\Gamma\to
\otimes^2$ defined, for an abelian group $A$, as $\gamma(a)\mapsto
a\otimes a,\ a\in A$.

Define the functor $\Phi$ as a kernel of $\Gamma\to \otimes^2$.
Then, for any abelian group $A$, there is a natural exact sequence
$$
0\to \Phi(A)\to \Gamma(A)\to A\otimes A\to \Lambda^2(A)\to 0
$$
where $\Lambda^2$ is the exterior square. One can easily check
that, for any pair of abelian groups $A,B$, the (bi)natural map
between the cross-effects of the functors $\Gamma$ and $\otimes^2$
$$
\Gamma(A|B)=A\otimes B \to \otimes^2(A|B)=A\otimes B\oplus
B\otimes A
$$
is a monomorphism. From this property together with the above
description of the values of $\Gamma$ for cyclic groups follows
that $\Phi$ is a 2-torsion functor, i.e. {\it for any $A$ and
$a\in \Phi(A)$, $2a=0$ in $\Phi(A)$}.

\subsection{Kernel of the Hurewicz homomorphism}
\begin{proposition}\label{h2prop}
For any connected space $X$, the kernel of the Hurewicz
homomorphism
$$
h_2\Omega: \pi_2(\Omega X)\to H_2(\Omega X)
$$
is a 2-torsion subgroup of $\pi_2(\Omega X)=\pi_3(X).$
\end{proposition}
Observe that, the statement about the third Hurewicz homomorphism
obviously is not true without taking loops. For any odd prime $p$,
the Moore space $P^3(p)$ has $\pi_3(P^3(p))=\mathbb Z/p$ and
$H_3(P^3(p))=0$.

\begin{proof} First consider the case of a simply-connected space $Y$. Let
$GY$ be the simplicial Kan loop construction. The following
diagram of fibrations
$$
\xyma{[GY,GY]\ar@{>->}[r] \ar@{->}[d] & GY \ar@{->>}[r]
\ar@{->}[d] & (GY)_{ab}\ar@{->}[d]\\ \Delta^2(GY)\ar@{>->}[r] &
\mathbb Z[GY]\ar@{->>}[r] & \mathbb Z[GY]/\Delta^2(GY)}
$$
induces the commutative diagram of homotopy groups
\begin{equation}\label{h4d}
\xyma{H_4(Y) \ar@{=}[d] \ar@{->}[r] & \pi_2([GY,GY])\ar@{->}[r]
\ar@{->}[d] & \pi_2(\Omega Y)\ar@{->>}[r] \ar@{->}[d]^{h_2\Omega}& H_3(Y)\ar@{=}[d]\\
H_4(Y) \ar@{->}[r] & \pi_2(\Delta^2(GY))\ar@{->}[r] & H_2(\Omega
Y) \ar@{->>}[r] & H_3(Y) }
\end{equation}
A simple analysis of connectivity of the simplicial groups $[GY,
GY]$ and $\Delta^2(GY)$ shows that there are natural isomorphisms
\begin{align*}
& \pi_2([GY,GY])=\pi_2([GY,GY]/[[GY,GY],GY])=\pi_2(\Lambda^2((GY)_{ab})),\\
&
\pi_2(\Delta^2(GY))=\pi_2(\Delta^2(GY)/\Delta^3(GY))=\pi_2((GY)_{ab}\otimes
(GY)_{ab}).
\end{align*}
The derived functors of $\Lambda^2$ and $\otimes^2$  are
well-known in a general situation (see, for example,
\cite{BauesPira}). We obtain the following natural diagram
$$
\xyma{\pi_2([GY,GY])\ar@{->}[r]\ar@{=}[d] & \pi_2(\Delta^2(GY))\ar@{=}[d]\\
\Gamma(H_2(Y))\ar@{->}[r] & H_2(Y)\otimes H_2(Y)}
$$
The left hand isomorphism in the last diagram is a reformulation
of the result due to Whitehead \cite{W}, the right hand
isomorphism follows from the Kunneth formula. Now lemmas
\ref{Lemma_pushout} and \ref{Lemma_pullback} imply that, the
diagram (\ref{h4d}) can be extended to the following diagram
$$
\xyma{ & \Phi(H_2(Y)) \ar@{>->}[d] \ar@{->>}[r] & K\ar@{>->}[d]\\
H_4(Y) \ar@{=}[d] \ar@{->}[r] & \Gamma(H_2(Y))\ar@{->}[d]
\ar@{->}[r] & \pi_3(Y)\ar@{->>}[r]\ar@{->}[d]^{h_2\Omega} & H_3(Y)\ar@{=}[d] \\
H_4(Y) \ar@{->}[r] & H_2(Y)\otimes H_2(Y)\ar@{->>}[d] \ar@{->}[r]
& H_2(\Omega Y) \ar@{->>}[d] \ar@{->>}[r] & H_3(Y)\\ &
H_2(Y)\wedge H_2(Y)\ar@{=}[r] & H_2(\pi_1(\Omega Y))}
$$
where the upper horizontal map is an epimorphism. Since the group
$\Phi(H_2(Y))$ is 2-torsion, the kernel $K$ of the Hurewicz
homomorphism also is 2-torsion and the needed statement is proved.

Now consider the case of arbitrary connected space $X$. Consider
its universal cover $\tilde X\to X$. The needed statement follows
from the diagram
$$
\xyma{\pi_2(\Omega \tilde X)\ar@{=}[r] \ar@{->}[d]^{h_2\Omega} & \pi_2(\Omega X)\ar@{->}[d]^{h_2\Omega}\\
H_2(\Omega \tilde X)\ar@{>->}[r] & H_2(\Omega X)}
$$
and the above proof of the statement for the simply-connected
case.
\end{proof}

\section{Cubes of simplicial non-unital rings and their crossed
cubes}\label{section3}

\subsection{Cubes of fibrations and fibrant cubes}
Set $\langle n\rangle=\{1,\dots, n\}.$ By a ring we assume a
\underline{\bf non-unital} ring, and by a ring homomorphism we
assume a non-unital ring homomorphism.

Consider the category of simplicial rings (s.r.)  as a model
category, whose weak equivalences are weak equivalences of
underlying simplicial sets and fibrations are level-wise
surjective homomorphisms (see Ch.2 $\S4$ \cite{Quillen}). Then a
fibration sequence in ${\sf sRng}$ is isomorphic to a sequence of
the form
\begin{equation}\label{eq_fibr}
I\to R  \twoheadrightarrow R/I,
\end{equation}
 where $I$ is an ideal of the simplicial ring $R.$

Consider the ordered sets $\{0, 1\}$ and  $\{-1, 0, 1\}$ as
categories in a usual way. Let $\mathcal F:\{-1,0,1\}^n\to {\sf
sRng}$ be a functor.  For two disjoint subsets
$\alpha,\beta\subseteq \langle n \rangle $ (i.e. $\alpha \cap
\beta =\emptyset$) we put
$$\mathcal F(\alpha,\beta)=\mathcal F(i_1,\dots,i_n),$$
where $\alpha=\{k\mid i_k=1\}$ and $\beta=\{k\mid i_k=-1\}.$ Then,
if $\alpha'\supseteq \alpha$ and $\beta' \subseteq \beta,$  we
have a map
$$\mathcal F(\alpha,\beta)\longrightarrow \mathcal F(\alpha',\beta').$$

 An {\bf $n$-cube of fibrations} of s.r. (see \cite[1.3]{Loday}) is a functor $\mathcal F:\{-1,0,1\}^n\to {\sf sRng}$ such that
for any disjoint subsets $\alpha,\beta\subseteq \langle n \rangle$
and $k\in \langle n \rangle\setminus (\alpha\cup \beta)$ we have a
fibration sequence
$$\mathcal F(\alpha,\beta\cup\{k\})\to \mathcal F(\alpha,\beta)\twoheadrightarrow \mathcal F(\alpha\cup\{k\},\beta).$$ If $n=2$ it is a $3\times 3$ square whose rows and columns are fibration sequences:
$$\xymatrix{
\mathcal F(\emptyset,\langle 2\rangle)\ar[r]\ar[d] & \mathcal F(\emptyset,\{2\})\ar[r]\ar[d] & \mathcal F(\{1\},\{2\})\ar[d]\\
\mathcal F(\emptyset,\{1\})\ar[r]\ar[d] & \mathcal F(\emptyset,\emptyset)\ar[r]\ar[d] & \mathcal F(\{1\},\emptyset)\ar[d] \\
\mathcal F(\{2\},\{1\})\ar[r] & \mathcal F(\{2\},\emptyset)\ar[r]
& \mathcal F(\langle 2\rangle,\emptyset). }$$

An {\bf $n$-cube} of s.r. is a functor $\mathcal R:\{0,1\}^n\to
{\sf sRng}.$ We set
$$\mathcal R(\alpha)=\mathcal R(i_1,\dots,i_n),$$
where $\alpha=\{k\mid i_k=1\}.$
 It is easy to see that $\{0,1\}^n$ is a direct category, and hence, it is a Reedy category. It follows that there is a natural model structure on the category of $n$-cubes of simplicial rings, called Reedy model structure \cite{Hovey}, \cite{Hirschhorn}. Weak equivalences of $n$-cubes in this model structure are defined level-wise.  An $n$-cube $\mathcal R$ is {\bf fibrant} in this model structure if and only if the map
$$\mathcal R(\alpha)\longrightarrow {\sf lim}_{\alpha'\supsetneq\: \alpha}\: \mathcal R(\alpha')$$
is a fibration of s.r. for any $\alpha\subseteq \langle n
\rangle$.

Let $\mathcal R$ be an $n$-cube of s.r. We consider the functor
${\sf E}_{\rm fib}(\mathcal R):\{-1,0,1\}^n\to {\sf sRng}$ given
by
\begin{equation}\label{eq_E_fib}
{\sf E}_{\rm fib}(\mathcal R)(\alpha,\beta)={\rm Ker}\left(
\mathcal R(\alpha)\to \prod_{i\in \beta} \mathcal R(\alpha\cup
\{i\}) \right)
\end{equation}
with obvious morphisms.

\begin{lemma}\label{Lemma_fibrant_cubes}Let $\mathcal R$ be an  $n$-cube of s.r. Then the following statements are equivalent.
\begin{enumerate}
\item $\mathcal R$ is fibrant. \item $\mathcal R$ can be embedded
into an $n$-cube of fibrations. \item  ${\sf E}_{\rm fib}(\mathcal
R)$ is an $n$-cube of fibrations.
\end{enumerate}
Moreover, if $\mathcal R$ is a fibrant $n$-cube of s.r., ${\sf
E}_{\rm fib}(\mathcal R)$ is the unique (up to unique isomorphism
that respects the embeddings) $n$-cube of fibrations to which
$\mathcal R$ can be embedded.
\end{lemma}
\begin{proof}  If $d\geq 0,$ $k\notin \alpha$ and $r\in \mathcal R(\alpha)_d,$ we denote by $r^k$ the image of $r$ in $\mathcal R(\alpha\cup \{k\})_d.$
Assume that $\alpha,\beta\subseteq \langle n \rangle$ are disjoint
sets and $d\geq 0$. An {\bf $(\alpha,\beta)$-collection} is a
collection $(r_i)\in \prod_{i\in \beta} \mathcal R(\alpha\cup
\{i\})$ such that $r_i^j=r_j^i$ for any $i,j\in \beta.$ A {\bf
lifting} of an  $(\alpha,\beta)$-collection $(r_i)$ is an element
$r\in \mathcal R(\alpha)_d$ such that $r_i=r^i.$
  It is easy to see that the ring ${\sf lim}_{\alpha'\supsetneq\: \alpha}\: \mathcal R(\alpha')_d$ consist of  $(\alpha,\alpha^{\rm c})$-collections, where $\alpha^{\sf c}=\langle n \rangle \setminus \alpha.$ Then $\mathcal R$ is fibrant if and only if for any  $(\alpha,\alpha^{\sf c})$-collecion there exists a lifting. We claim that if $\mathcal R$ is fibrant then for any disjoint $\alpha,\beta$ and any  $(\alpha,\beta)$-collection there exist a lifting. The prove is by induction on $|\langle n \rangle\setminus (\alpha\cup \beta)|.$ If it is equal to $0,$ we done. Assume that $(r_i)$ is an $(\alpha,\beta)$-collection. Consider any $j\in  \langle n \rangle\setminus (\alpha\cup \beta)$ and the $(\alpha\cup \{j\},\beta)$-collection $(r^j_i).$
By induction hypothesis we have a lifting $r_j\in \mathcal
R(\alpha\cup \{j\})$ of $(r_i^j).$ Then we get a
$(\alpha,\alpha^{\sf c})$-collection $(r_i)_{i\in \alpha^{\sf
c}},$ whose lifting is the lifting of the original
$(\alpha,\beta)$-collection $(r_i).$ Therefore $\mathcal R$ is
fibrant if and only if any $(\alpha,\beta)$-collection has a
lifting.

(1)$\Rightarrow$(3). Assume that $\mathcal R$ is fibrant and prove
that ${\sf E}_{\rm fib}(\mathcal R)$ is an $n$-cube of fibrations.
Consider the diagram with exact rows
$$\xymatrix{0\ar[r] & {\sf E}_{\rm fib}(\mathcal R)(\alpha,\beta\cup\{k\})\ar[r]\ar[d] & \mathcal R(\alpha) \ar[r]\ar[d] & \prod_{i\in \beta\cup \{k\}}\mathcal R(\alpha\cup\{i\}) \ar[d]\\
 0 \ar[r] & {\sf E}_{\rm fib}(\mathcal R)(\alpha,\beta)\ar[r]\ar[d] & \mathcal R(\alpha) \ar[r]\ar[d] & \prod_{i\in \beta}\mathcal R(\alpha\cup\{i\}) \ar[d] \\
0\ar[r] & {\sf E}_{\rm fib}(\mathcal
R)(\alpha\cup\{k\},\beta)\ar[r] & \mathcal R(\alpha\cup\{k\})
\ar[r] & \prod_{i\in \beta}\mathcal R(\alpha\cup\{i,k\}). }$$ We
have to prove that the left column is a short exact sequence. The
only non-obvious thing is that ${\sf E}_{\rm fib}(\mathcal
R)(\alpha,\beta)_d\to {\sf E}_{\rm fib}(\mathcal
R)(\alpha\cup\{k\},\beta)_d $ is surjective. Consider any $r_k\in
{\sf E}_{\rm fib}(\mathcal R)(\alpha\cup\{k\},\beta)_d$ and denote
$r_i=0$ for $i\in \beta.$ Then $(r_i)$ is a $(\alpha,\beta\cup
\{k\})$-collection, whose lifting is a preimage of $r_k.$

(3)$\Rightarrow$(1). Assume that ${\sf E}_{\rm fib}(\mathcal R)$
is an $n$-cube of fibrations and prove that $\mathcal R $ is
fibrant. We need to prove that for any $(\alpha,\beta)$-collection
there exists a lifting. Prove it by induction on $|\beta|.$ If
$\beta=\emptyset,$ it is obvious. Assume that it holds for $\beta$
and prove it for $\beta'=\beta\cup\{k\},$ where $k\in \langle
n\rangle\setminus (\alpha\cup \beta).$ Consider an
$(\alpha,\beta\cup \{k\})$-collection $(r_i)_{i\in
\beta\cup\{k\}}.$ By induction hypotheses its
$(\alpha,\beta)$-subcollection $(r_i)_{i\in \beta}$ has a lifting
$\bar r\in \mathcal R(\alpha)_d.$  Then $r_k-\bar r^k\in {\sf
E}_{\rm fib}(\alpha\cup \{k\},\beta)_d.$ Since the map ${\sf
E}_{\rm fib}(\mathcal R)(\alpha,\beta)_d\to {\sf E}_{\rm
fib}(\mathcal R)(\alpha\cup\{k\},\beta)_d $ is surjective we get a
preimage $\hat r\in {\sf E}_{\rm fib}(\mathcal R)(\alpha,\beta)_d$
such that $\hat r^k=r_k-\bar r^k.$ Then $r=\hat r+\bar r$ is a
lifting of the $(\alpha,\beta\cup \{k\})$-collection $(r_i)_{i\in
\beta\cup\{k\}}.$

(3)$\Rightarrow$(2). Obvious.

(2)$\Rightarrow$(3). Assume that $\mathcal R$ is embedded into an
$n$-cube of fibrations $\mathcal F.$ Replacing $\mathcal F$ by
isomorphic one, we can assume that the fibres are identical
embeddings. Prove that $\mathcal F(\alpha,\beta)={\sf E}_{\rm
fib}(\alpha,\beta)$ by induction on $|\beta|.$ If
$\beta=\emptyset,$ then   $\mathcal F(\alpha,\emptyset)=\mathcal
R(\alpha)={\sf E}_{\rm fib}(\alpha,\emptyset).$ Assume that it
holds for $\beta$ and prove it for $\beta'=\alpha \cup \{k\}.$ By
induction hypothesis we have a commutative diagram
$$\xymatrix{&0\ar[d]&& \\ & \mathcal F(\alpha,\beta\cup\{k\})\ar[r]\ar[d] & \mathcal R(\alpha) \ar[r]\ar[d] & \prod_{i\in \beta\cup \{k\}}\mathcal R(\alpha\cup\{i\}) \ar[d]\\
 0 \ar[r] & {\sf E}_{\rm fib}(\mathcal R)(\alpha,\beta)\ar[r]\ar[d] & \mathcal R(\alpha) \ar[r]\ar[d] & \prod_{i\in \beta}\mathcal R(\alpha\cup\{i\}) \ar[d] \\
0\ar[r] & {\sf E}_{\rm fib}(\mathcal
R)(\alpha\cup\{k\},\beta)\ar[r] & \mathcal R(\alpha\cup\{k\})
\ar[r] & \prod_{i\in \beta}\mathcal R(\alpha\cup\{i,k\}), }$$
where the right column is a fibration sequence. Using that
$\mathcal F(\alpha,\beta\cup\{k\})_d={\rm Ker}({\sf E}_{\rm
fib}(\alpha,\beta)_d\to {\sf E}_{\rm
fib}(\alpha\cup\{k\},\beta)_d)$ it is easy to deduce from the
diagram that $\mathcal F(\alpha,\beta\cup\{k\})_d={\rm
Ker}(\mathcal R(\alpha)_d\to \prod_{i\in \beta\cup \{k\}}\mathcal
R(\alpha \cup \{i\})_d)={\sf E}_{\rm fib}(\mathcal
R)(\alpha,\beta\cup \{k\})_d.$
\end{proof}

\subsection{Cubes of fibrations and good tuples of  ideals}

An {\bf $n$-tuple of ideals} of a s.r. is a tuple
$I=(R;I_1,\dots,I_n),$ where $R$ is a simplicial ring and $I_i$
are (simplicial) ideals of $R.$ For $\beta \subseteq \langle n
\rangle$ we set
$$I(\beta)=\bigcap_{i\in \beta} I_i.$$
An $n$-tuple of ideals $I$ is said to be {\bf good} if for any
disjoint subsets $\alpha,\beta \subset \langle n \rangle$ and
$k\in \langle n\rangle\setminus (\alpha \cup \beta)$ the following
equality holds
\begin{equation}\label{eq_good}
I(\beta\cup \{k\})\cap \left(\sum_{i\in \alpha} I(\beta\cup
\{i\})\right)=\sum_{i\in \alpha} I(\beta\cup \{k,i\}).
\end{equation}
It easy to check that any $2$-tuple of ideals is always good. But
a $3$-tuple of ideals is good if and only if for any $i,j,k\in
\{1,2,3\}$ the following holds $I_i\cap (I_j+I_k)=I_i\cap
I_j+I_i\cap I_k.$

For an $n$-tuple of ideals $I=(R;I_1,\dots,I_n)$ we consider the
functor ${\sf E}_{\rm idl}(I):\{-1,0,1\}^n\to {\sf sRng}$ given by
\begin{equation}\label{eq_E_idl}
{\sf E}_{\rm idl}(\mathcal
I)(\alpha,\beta)=\frac{I(\beta)}{\sum_{i\in \alpha} I(\beta\cup
\{i\})}
\end{equation}
with obvious morphisms. For example ${\sf E}_{\rm idl}(R;I,J)$
looks as follows:
$$\xymatrix{
I\cap J\ar[r]\ar[d] & I\ar[r]\ar[d] & I/(I\cap J)\ar[d]\\
J\ar[r]\ar[d] & R\ar[r]\ar[d] & R/J\ar[d] \\
J/(I\cap J)\ar[r] & R/I\ar[r] & R/(I+J). }$$ Note that this
definition can be rewritten in a way dual to \eqref{eq_E_fib}:
${\sf E}_{\rm idl}(I)(\alpha,\beta )={\rm
Coker}\left(\coprod_{i\in \alpha} I(\beta \cup \{i\})\to I(\beta)
\right).$

For an $n$-cube cube of fibrations $\mathcal F$ we consider an
$n$-tuple of ideals
\begin{equation}\label{eq_T_idl}
{\sf T}_{\rm idl}(\mathcal F)=(R;I_1,\dots,I_n), \ \
R=\mathcal{F}(\emptyset,\emptyset),\ \  I_i={\rm Ker}(\mathcal
F(\emptyset,\emptyset)\to \mathcal{F}(\{i\},\emptyset)).
\end{equation}

\begin{lemma}\label{Lemma_good_ideals}Let $I$ be an  $n$-tuple of ideals of a s.r. Then the following statements are equivalent.
\begin{enumerate}
\item $I$ is good; \item $I={\sf T}_{\rm idl}(\mathcal F)$ for
some  $n$-cube of fibrations $\mathcal F$; \item  ${\sf E}_{\rm
idl}(I)$ is an $n$-cube of fibrations.
\end{enumerate}
Moreover, if $I$ is good, ${\sf E}_{\rm idl}(I)$ is the unique (up
to unique isomorphism that respects the equalities) $n$-cube of
fibrations such that $I={\sf T}_{\rm idl}({\sf E}_{\rm idl}(I))$.
\end{lemma}
\begin{proof}
(1)$\Leftrightarrow$(3). Let $\alpha,\beta\subset  \langle n
\rangle$ be disjoint sets and $k\in \langle n \rangle \setminus
(\alpha\cup \beta).$ Consider the sequence
$$\frac{I(\beta\cup\{k\})}{\sum_{i\in \alpha}I(\beta\cup\{i,k\})}\longrightarrow \frac{I(\beta)}{\sum_{i\in \alpha}I(\beta\cup\{i\})}\longrightarrow \frac{I(\beta)}{\sum_{i\in \alpha\cup\{k\}}I(\beta\cup\{i\})}.$$ It is easy to see that the right hand map is en epimorphism and that it is exact in the middle term. Moreover, the left hand homomorphism is a monomorphism if and only if \eqref{eq_good} holds.

(3)$\Rightarrow$(2). Follows from the equality ${\sf T}_{\rm
idl}({\sf E}_{\rm idl}(I))=I.$

(2)$\Rightarrow$(3). Assume that $I={\sf T}_{\rm idl}(\mathcal F)$
for some  $n$-cube of fibrations  $\mathcal F.$ Replacing
$\mathcal F$ by isomorphic one, we can assume that the fibres are
identical embeddings. First we prove that $I(\beta)=\mathcal
F(\emptyset,\beta).$  The prove is by induction on $|\beta|.$ If
$|\beta|=0,1$ it is obvious. Assume that $|\beta|\geq 2$ and fix
two distinct elements $k,l\in \beta.$ By induction hypothesis we
have $\mathcal F(\emptyset,\beta\setminus \{k\})=I(\beta\setminus
\{k\})$ and $\mathcal F(\emptyset,\beta\setminus
\{l\})=I(\beta\setminus \{l\}).$ Consider the diagram
 $$\xymatrix{0\ar[r] &
\mathcal F(\emptyset,\beta)\ar[r]\ar[d] & I(\beta\setminus \{k\})\ar[r]\ar[d] & \mathcal F(\{k\},\beta)\ar[d]\ar[r] & 0\\
0\ar[r] &I(\beta\setminus \{l\})\ar[r] & I({\beta\setminus\{k,l\}})\ar[r] & \mathcal F(\{k\},{\beta\setminus\{k,l\}})\ar[r] & 0, \\
}$$ whose rows are short exact sequences. Since the left square
consists of monomorphisms and the map on the cokernels $\mathcal
F(\{k\},\beta)\to \mathcal  F(\{k\},{\beta\setminus\{k,l\}})$ is a
monomorphism, by Lemma \ref{Lemma_pullback}  we get that the left
square is a pulback square. Hence, $\mathcal
F(\emptyset,\beta)=I(\beta\setminus \{k\}) \cap I(\beta\setminus
\{l\})=I(\beta).$

So we have $\mathcal F(\emptyset,\beta)={\sf E}_{\rm
idl}(\emptyset,\beta).$ Further we prove by induction on $
|\alpha|$ that there is a unique isomorphism $\mathcal
F(\alpha,\beta)\cong {\sf E}_{\sf idl}(I)(\alpha,\beta)$ that
satisfies commutation properties and lifts this equality for
$\alpha=\emptyset$. Assume that this holds for $\alpha$ and prove
it for $\alpha\cup\{k\}.$ By the induction hypothesis we have that
$\mathcal{F}(\alpha,\beta\cup\{k\})=\frac{I(\beta\cup\{k\})}{\sum_{i\in
\alpha}I(\beta\cup\{i,k\})}$ and
$\mathcal{F}(\alpha,\beta)=\frac{I(\beta)}{\sum_{i\in
\alpha}I(\beta\cup\{i\})}.$ It follows that there is a short exact
sequence
$$0\to \frac{I(\beta\cup\{k\})}{\sum_{i\in \alpha}I(\beta\cup\{i,k\})}\longrightarrow \frac{I(\beta)}{\sum_{i\in \alpha}I(\beta\cup\{i\})}\longrightarrow \mathcal F(\alpha\cup\{k\},\beta)\to 0,$$ which induces the required isomorphism  $\mathcal F(\alpha,\beta)\cong \frac{I(\beta)}{\sum_{i\in \alpha\cup\{k\}}I(\beta\cup\{i\})}.$
\end{proof}

\subsection{Three equivalent categories}
Consider the truncation functor
\begin{equation}\label{eq_T_fib}
{\sf T}_{\rm fib}:\text{(}n\text{-cubes of fibrations of
s.r.})\longrightarrow \text{(fibrant }n\text{-cubes of s.r.)}
\end{equation}
that induced by the embedding $\{0,1\}^n\subset \{-1,0,1\}^n.$
Lemma \ref{Lemma_fibrant_cubes} implies that this functor is well
defined.

\begin{proposition}The functors
\begin{equation}
\xymatrix{
\text{\rm (fibrant } n\text{\rm-cubes of s.r.)}\ar@<+10pt>[d]^{\sf E_{\rm fib}}_{\sim\hspace{1.5mm}}  \\
\text{\rm(}n\text{\rm-cubes of fibrations of s.r.)}\ar@<+10pt>[d]^{{\sf T}_{\rm idl}}_{\sim\hspace{1.5mm}}\ar@<+10pt>[u]^{{\sf T}_{\rm fib}}\\
\text{\rm(good }n\text{\rm-tuples of ideals of s.r.)}\ar@<+10pt>[u]^{{\sf E}_{\rm idl}}\\
}
\end{equation}
given by \eqref{eq_E_fib}, \eqref{eq_E_idl},\eqref{eq_T_fib} and
\eqref{eq_T_idl} define mutually invert equivalences of
categories.
\end{proposition}
\begin{proof}
The equalities ${\sf T}_{\rm fib}{\sf E}_{\rm fib}={\sf Id},$
${\sf T}_{\rm idl}{\sf E}_{\rm idl}={\sf Id}$ are obvious. The
isomorphisms ${\sf E}_{\rm fib}{\sf T}_{\rm idl}\cong {\sf Id}$
and ${\sf E}_{\rm fib}{\sf T}_{\rm idl}\cong {\sf Id}$  follow
from Lemma \ref{Lemma_fibrant_cubes} and Lemma
\ref{Lemma_good_ideals}.
\end{proof}

Consider the functor
\begin{equation}\label{eq_fibre}
{\sf Fibre}:\text{\rm (fibrant } n\text{\rm-cubes of s.r.)}
\overset{\sim}\longrightarrow \text{\rm(good }n\text{\rm-tuples of
ideals of s.r.)},
\end{equation}
given by ${\sf Fibre}(\mathcal R)=(R;I_1,\dots, I_n),$ where
$R=\mathcal R(\emptyset)$ and  $I_i={\rm Ker}(\mathcal
R(\emptyset)\to \mathcal R(\{i\})).$

\begin{corollary}The functor \eqref{eq_fibre} is an  equivalence of categories.
\end{corollary}

\subsection{Crossed cubes of cubes of simplicial rings}
 Following Ellis \cite{Ellis} we define
a {\bf crossed $n$-cube of rings} $\{R_\beta\}$ as a family of
rings, where $\beta\subseteq \langle n \rangle$ together with
homomorphisms $\mu_i:R_\beta \to R_{\beta\setminus \{i\}}$ and
$h:R_\beta\otimes R_{\beta'} \to R_{\beta\cup \beta'}$ such that
for $a,a'\in R_\beta,$ $b,b'\in R_{\beta'},$ $c\in R_{\beta''}$
and $i,j\in \langle n \rangle$ such that
\begin{itemize}
\item $\mu_i a=a$ if $i\notin \beta;$ \item
$\mu_i\mu_ja=\mu_j\mu_ia;$ \item $\mu_ih(a\otimes
b)=h(\mu_ia\otimes b)=h(a\otimes \mu_ib);$ \item $h(a\otimes
b)=h(\mu_ia\otimes b)=h(a\otimes \mu_ib)$ if $i\in \beta\cap
\beta';$ \item $h(a\otimes a')=aa';$
\end{itemize}
with the assotiative property:
\begin{itemize}
\item $h(h(a\otimes b)\otimes c)=h(a\otimes h(b\otimes c)).$
\end{itemize}
Morphisms of crossed $n$-cubes are defined obviously. Consider the
functor
\begin{equation}\label{eq_pi0}
\pi_0: \text{(}n\text{-tuples of ideals of s.r.)} \xrightarrow{\ \
\ \ } \text{(crossed } n\text{-cubes of r.)}
\end{equation}
that sends $I$ to $\{R_\beta\},$ where $R_\beta=\pi_0I(\beta),$
$\mu_i=\pi_0(I(\beta)\hookrightarrow I(\beta\setminus \{i\}))$
 and $h:R_{\beta}\otimes R_{\beta'} \to R_{\beta\cup \beta'}$ is the
 composition of the isomorphism $\pi_0I(\beta)\otimes \pi_0I({\beta'})\cong \pi_0(I(\beta)\otimes\mathcal  I({\beta'}))$
 and the map $\pi_0(I_{\beta} \otimes I_{\beta'}\to I_{\beta\cup {\beta'}}).$ It is easy to check that $\{R_\beta\}$ is a crossed $n$-cube of rings.

In the Reedy model structure on the category of $n$-cubes there is
a functorial fibrant replacement
$$\gamma:\mathcal R\overset{\sim}\longrightarrow \overline{\mathcal R},$$
where $\overline{\mathcal R}$ is a fibrant $n$-cube and $\gamma$
is a weak equivalence and a cofibration. Then consider the functor
\begin{equation}\label{pi}\Pi:\text{(}n\text{-cubes of s.r.)} \longrightarrow
\text{(crossed }n\text{-cubes of rings)},\end{equation} given by
$\Pi(\mathcal R)=\pi_0({\sf Fibre}(\overline{\mathcal R})),$ which
is analogue of the  one constructed in \cite{Brown_Loday} for
simplicial rings. Analysing the definition of $\Pi$ we get the
following. If $\mathcal R$ is an $n$-cube of s.r. we can embed it
into a $n$-cube of {\bf homotopy} fibration sequences $\mathcal F$
in the {\bf homotopy category} of simplicial rings by taking
homotopy fibres of all arrows, and then
$$\Pi(\mathcal R)_\beta=\pi_0(\mathcal F(\emptyset,\beta)).$$

\section{Proof of theorem \ref{main}}

\subsection{Proof of theorem \ref{main}} Now suppose that $R, S, T$ are normal subgroups of a
group $G$. Let $X$ be a homotopy pushout of the following diagram
of classifying spaces: \begin{equation}\label{pd}{\tiny { \xyma{&
& BG \ar@{->}[rr] \ar@{->}[ldd]
\ar@{-}[dd] && B(G/R) \ar@{->}[ddd] \ar@{->}[ldd]\\ \\
& B(G/S) \ar@{->}[ddd] \ar@{->}[rr] & \ar@{->}[d] & B(G/RS) \ar@{->}[ddd]\\
& & B(G/T) \ar@{-}[r] \ar@{->}[ldd] & \ar@{->}[r] & B(G/RT)
\ar@{->}[ldd]\\ & & \\ & B(G/ST) \ar@{->}[rr] & & X  } }}
\end{equation}
There is a natural isomorphism of groups (see \cite{EM}):
$$
\pi_1(X)\simeq G/RST.
$$
Certain higher homotopy groups of $X$ are described in \cite{EM}.
In particular, there is a natural isomorphism of
$\pi_1(X)$-modules:
\begin{equation}\label{ellis}
\pi_3(X)\simeq \frac{R\cap S\cap T}{\|R,S,T\|}
\end{equation}
where the action of $\pi_1(X)\simeq G/RST$ on the right hand side
of (\ref{ellis}) is viewed via conjugation in $G$. Recall the idea
of the proof from \cite{EM}. Extend the above pushout to the cube
of fibrations which have 27 spaces. The $\pi_1$ of the complement
to the pushout in the cube of fibrations
$$
{\tiny { \xyma{& & \pi_1({\sf upper}\ {\sf corner}) \ar@{->}[rr]
\ar@{->}[ldd]
\ar@{-}[dd] && S\cap T \ar@{->}[ddd] \ar@{->}[ldd]\\ \\
& R\cap T \ar@{->}[ddd] \ar@{->}[rr] & \ar@{->}[d] & T \ar@{->}[ddd]\\
& & R\cap S \ar@{-}[r] \ar@{->}[ldd] & \ar@{->}[r] & S \ar@{->}[ldd]\\
& & \\ & R \ar@{->}[rr] & & G  } }}
$$
is a crossed cube of groups. Moreover, it is a {\it universal}
crossed cube of groups (see \cite{EM} for definition and
discussion of the universality). One can realize the pushout
diagram as a diagram of simplicial groups. The functor of group
rings ${\mathbb Z}[-]: {\sf groups}\to {\sf group}\ {\sf rings}$
sends pushouts to the pushouts in the category of simplicial
rings. Extending this pushout diagram to a cube of homotopy
fibration sequences in the category of simplicial rings and taking
$\pi_0$ of the complement part as in (\ref{pi}) we obtain the
crossed cube of rings
$${\small
{ \xyma{& & \pi_0({\sf upper}\ {\sf corner}) \ar@{->}[rr]^{\mu_1}
\ar@{->}[ldd]^{\mu_3}
\ar@{-}[dd]^{\mu_2} && {\bf s}\cap {\bf t} \ar@{->}[ddd] \ar@{->}[ldd]\\ \\
& {\bf r}\cap {\bf s} \ar@{->}[ddd] \ar@{->}[rr] & \ar@{->}[d] & {\bf s} \ar@{->}[ddd]\\
& & {\bf r}\cap {\bf t} \ar@{-}[r] \ar@{->}[ldd] & \ar@{->}[r] &
{\bf t} \ar@{->}[ldd]\\ & & \\ & {\bf r} \ar@{->}[rr] & & \mathbb
Z[G] } }}
$$
Now we observe that
$$
\|{\bf r},{\bf s},{\bf t}\|\subseteq {\rm Im}(\mu_i),\ i=1,2,3.
$$
This follows from the properties of crossed cubes
$$\mu_ih(a\otimes b)=h(\mu_ia\otimes b)=h(a\otimes \mu_ib),\ i=1,2,3$$
for $a\in {\bf r\cap \bf s}, b\in {\bf t}$ and other choices of
ideals. Applying homotopy exact sequences of fibrations three
times, and comparing them for simplicial groups and group rings,
we obtain the needed commutative diagram
$$
\xyma{\frac{R\cap S\cap T}{\|R,S,T\|}\ar@{->}[r]\ar@{=}[d] &
\frac{\bf
r\cap \bf s\cap \bf r}{\|\bf r,\bf s,\bf t\|}\ar@{->}[d]\\
\pi_2(\Omega X)\ar@{->}[r]^{h_2\Omega}& H_2(\Omega X)}
$$
which considered together with proposition \ref{h2prop} imply the
needed statement. Theorem \ref{main} follows.\ \ $\Box$

\subsection{The case of two subgroups} Observe that, in the case of two normal subgroups $R,S$
in $G$ is much simpler than the above case. In this case, one has
a square of fibrations (in the category of simplicial rings)
$$
\xyma{T\ar@{->}[r] \ar@{->}[d] \ar@{->}[r] &
{\bf r} \ar@{->}[d] \ar@{->}[r] & fib_2\ar@{->}[d]\\
{\bf s}\ar@{->}[r]\ar@{->}[d] & \mathbb Z[G] \ar@{->}[r] \ar@{->}[d] & \mathbb Z[G/S]\ar@{->}[d] \\
fib_1\ar@{->}[r] & \mathbb Z[G/R] \ar@{->}[r] & \mathbb Z[X]}
$$
such that
$$
\xyma{\pi_0(T) \ar@{->}[r]^{\mu_2} \ar@{->}[d]^{\mu_1} & {\bf r}\ar@{->}[d]\\
{\bf s}\ar@{->}[r] & \mathbb Z[G]}
$$
is a crossed square of rings. Since $\mu_1h(a\otimes b)=ab,\ a\in
{\bf s}, b\in {\bf r},$ and $\mu_1h(a\otimes b)=ab,\ a\in {\bf
s},b\in {\bf r}$, ${\bf rs+sr}\subseteq {\rm Im}(\mu_1)$ and ${\bf
rs+sr}\subseteq {\rm Im}(\mu_2)$. Comparing the picture for groups
and group rings we conclude that there is a commutative diagram
$$
\xyma{\frac{R\cap S}{[R,S]}\ar@{=}[d]\ar@{>->}[r] & \frac{{\bf
r}\cap {\bf s}}{\bf rs+sr}\ar@{->}[d]\\ \pi_1(\Omega
X)\ar@{->}[r]^{h_1\Omega} & H_1(\Omega X)}
$$
Since, for any connected space $X$, the Hurewicz homomorphism
$\pi_1(\Omega X)\to H_1(\Omega X)$ is a monomorphism, we obtain
the following identification of the generalized dimension subgroup
$$
D(G, {\bf rs+sr})=[R,S].
$$
This gives a new proof of the result from \cite{BD}. This result can be generalized as follows.

Let $T$ be a normal subgroup of subgroup of $G$ and ${\bf
n}=(T-1)\mathbb Z[G]$. We obtain the following diagram
$$
\xyma{\frac{R\cap S}{[R,S][R\cap S,T]}\ar@{=}[d]\ar@{->}[r] &
\frac{{\bf r}\cap {\bf s}}{\bf rs+sr+({\bf r}\cap {\bf s}){\bf t}+{\bf t}({\bf r}\cap {\bf s})}\ar@{->}[d]\\
\pi_1(\Omega X)_{TRS/RS}\ar@{->}[r]^{(h_1\Omega)_{TRS/RS}} &
H_1(\Omega X)_{TRS/RS}.}
$$
Here the group $TRS/RS$ is considered as a subgroup of
$\pi_1(X)=G/RS.$ We show that $(h_1\Omega)_{TRS/RS}\colon
\pi_1(\Omega X)_{TRS/RS}\longrightarrow H_1(\Omega X)_{TRS/RS}$ is
a monomorphism for any normal subgroup $T$ of $G$. Let $G_*$ be a
simplicial group such that the geometric realization $|G_*|$ is
weakly homotopy equivalent to $\Omega X$. Observe that the action
of $\pi_1(X)\cong\pi_0(G_*)$ is induced by the conjugation action
of $G_0$ on $G_*$. Let $\tilde G_*$ be the path-connected
component of $G_*$ containing the identity element. From the
simplicial Postnikov system, there is a short exact sequence of
simplicial groups
$$
1\longrightarrow \tilde G_*\longrightarrow G_*\longrightarrow \pi_0(G)\longrightarrow 1,
$$
where $\pi_0(G)$ is the discrete simplicial group. Hence
$$
G_*=\coprod_{g\in \pi_0(G)} g\tilde G_*
$$
as a simplicial set. Let $\chi_h\colon G_*\to G_*, x\mapsto
hxh^{-1}$ be the conjugation action of $h\in G_0$ on $G_*$. Then
$\chi_h(g\tilde G_*)=hgh^{-1}\tilde G_*$ with
$$
\chi_h(gx)=(hgh^{-1})(hxh^{-1}).
$$
This implies that
$$
H_k(\Omega X)\cong H_*(G_*)\cong H_k(\tilde G_*)\otimes
\Z[\pi_0(G_*)]
$$
as modules over $\Z[\pi_0(G_*)]$ for $k\geq0$, where
$\Z[\pi_0(G_*)]$ acts diagonally on the tensor product $H_k(\tilde
G_*)\otimes \Z[\pi_0(G_*)]$.  It follows that $H_k(\Omega_0X)\cong
H_k(\tilde G_*)$ is a $\Z[\pi_0(G_*)]\cong
\Z[\pi_1(X)]$-equivariant summand of $H_k(\Omega X)$, where
$\Omega_0X$ is the path-connected component of $X$ containing the
basepoint. By taking $k=1$ with using the fact that $\pi_1(\Omega
X)\cong H_1(\Omega_0X)$, we have
 $(h_1\Omega)_{TRS/RS}\colon \pi_1(\Omega X)_{TRS/RS}\longrightarrow H_1(\Omega X)_{TRS/RS}$
 is a monomorphism for any normal subgroup $T$ of $G$. As a consequence, we obtain that
 \begin{equation}
\frac{R\cap S}{[R,S][R\cap S,T]}\longrightarrow  \frac{{\bf r}\cap
{\bf s}}{\bf rs+sr+({\bf r}\cap {\bf s}){\bf t}+{\bf t}({\bf
r}\cap {\bf s})}
\end{equation}
is a monomorphism for any normal subgroup $T\leq G$. We proved the
following
\begin{theorem}\label{modg}
For a group $G$ and its normal subgroups $R,S,T$,
$$
D(G,{\bf rs+sr}+({\bf r}\cap {\bf s}){\bf t}+{\bf t}({\bf r}\cap
{\bf s}))=[R,S][R\cap S,T].
$$
\end{theorem}

\section{Generalizations and examples}\label{lastsection}
\subsection{Simplicial groups} Let $G$ be a group and $R_1,\dots, R_n,\ n\geq 2$ its
normal subgroups. Denote
$$
\|R_1,\dots, R_n\|:=\prod_{I\cup J=\{1,\dots,n\}, I\cap
J=\emptyset}[\cap_{i\in I}R_i,\cap_{j\in J}R_j]
$$
Similarly, for a collection $\mathbf{a}_0,\dots, \mathbf{a}_n$ of
ideals in a ring $R$, denote
$$
\| \mathbf{a}_1,\dots, \mathbf{a}_n\|:=\sum_{I\cup
J=\{1,\dots,n\}, I\cap J=\emptyset}(\cap_{i\in
I}\mathbf{a}_i)(\cap_{j\in J}\mathbf{a}_j)+(\cap_{j\in
J}\mathbf{a}_j)(\cap_{i\in I}\mathbf{a}_i).
$$
It is easy to check that for arbitrary ideals ${\bf a}, {\bf b}$
of $\mathbb Z[G]$ we have $D({\bf a})D({\bf b}) \subseteq D({\bf
a}+{\bf b})$ and $[D({\bf a}),D({\bf b})]\subseteq D({\bf a}{\bf
b}+{\bf b}{\bf a}).$ Indeed, the first inclusion is obvious, and
the second follows from the equality
$g^{-1}h^{-1}gh-1=g^{-1}h^{-1}((g-1)(h-1)-(h-1)(g-1))$ for
arbitrary $g,h\in G.$ Therefore, for any $G$ and its normal
subgroups $R_1,\dots, R_n$,
$$
\|R_1,\dots, R_n\|\subseteq D(G,\| \mathbf{r}_1,\dots,
\mathbf{r}_n\|),
$$
where $\mathbf{r}_i=(R_i-1)\mathbb Z[G]$. One can ask how to
identify the generalized dimension quotient in the case of $>3$
subgroups or at least to describe its exponent. Here is an
approach for constructing of different examples.

If $\mathcal G$ is a simplicial group, we denote by $N\mathcal G$
its Moore complex. Cycles of this complex we denote by
$Z_n\mathcal G={\rm Ker}(d:N_n\mathcal G\to N_{n-1}\mathcal G),$
and boundaries by $B_n\mathcal G={\rm Im}(d:N_{n+1}\mathcal G\to
N_n\mathcal G).$ Then $\pi_n\mathcal G=Z_n\mathcal G/B_n\mathcal
G.$ The following theorem is Theorem B of
\cite{Castiglioni_Ladra}.
\begin{theorem}[{\cite{Castiglioni_Ladra}}]\label{Theorem_B} Let
$\mathcal G$ be a simplicial group and $n\geq 1$ such that
$\mathcal G_{n+1}$ is generated by degeneracies. Set $K_i:={\rm
Ker}(d_i:\mathcal G_n\to \mathcal G_{n-1}).$ Then
$$B_n\mathcal G=\|K_0,\dots, K_n\|,\hspace{1cm} \pi_n\mathcal
G=\frac{\bigcap_{i=0}^n K_i}{\|K_0,\dots, K_n\|}.$$
\end{theorem}
The following theorem is an analogue of the previous one for the
case of simplicial rings. It follows from Theorem A of
\cite{Castiglioni_Ladra}.
\begin{theorem}[{\cite{Castiglioni_Ladra}}]\label{Theorem_A}
Let $\mathcal R$ be a simplicial ring and $n\geq 1$ such that
$\mathcal R_{n+1}$ is generated by degeneracies as a ring. Set
${\mathbf{k}}_i:={\rm Ker}(d_i:\mathcal R_n\to \mathcal R_{n-1}).$
Then
$$B_n\mathcal R=\|\mathbf{k}_0,\dots,\mathbf{k}_n\|, \hspace{1cm}
\pi_n\mathcal R=\frac{\bigcap_{i=0}^n {\bf
k}_i}{\|\mathbf{k}_0,\dots,\mathbf{k}_n\|}.$$
\end{theorem}

For a simplicial group $\mathcal G$, the Hurewicz homomorphism
$h:\pi_n\mathcal G\to H_n\mathcal G$ for $n\geq 1$ is induced by
the map $\mathcal G\to \mathbb Z[\mathcal G]$ given by $g\mapsto
g-1.$ The following statement is a direct corollary of theorems
\ref{Theorem_B} and \ref{Theorem_A}.

\begin{proposition}\label{simple} Let $\mathcal G$ be a simplicial group and $n\geq 1$ such that $\mathcal G_{n+1}$ is generated by degeneracies.
Set
\begin{align*}
& K_i:={\rm Ker}(d_i:\mathcal G_n\to \mathcal G_{n-1})\\
& {\bf k}_i:={\rm Ker}(d_i: \mathbb Z[\mathcal G_n] \to \mathbb
Z[\mathcal G_{n-1}]).
\end{align*}
Then
$$\frac{D(\mathcal G_n,\|\mathbf{k}_0,\dots,\mathbf{k}_n\|)}{\|K_0,\dots, K_n\|} ={\rm Ker}(h_n: \pi_n\mathcal G \longrightarrow H_n\mathcal G),$$
 where $h_n$ is the $n$th Hurewicz homomorphism.
\end{proposition}

The generalized dimension subgroups as in proposition \ref{simple}
were considered in \cite{MPW} for the case of simplicial
Carlsson's constructions. The main example which we will consider
here is the $p$-Moore space $P^3(p)=S^2\cup_{p}e^2$ for a prime
$p\geq 2$. The lowest homotopy group of $P^3(p)$ which contains
$\mathbb Z/p^2$-summand is $\pi_{2p-1}P^3(p)$. This was proved in
\cite{CMN} for $p>3$, however, $\pi_3P^3(2)=\mathbb Z/4$,
$\pi_5P^3(3)=\mathbb Z/9$. Since all homology groups $H_*(\Omega
P^3(p))$ have exponent $p$, we have
$$
\mathbb Z/p\subseteq {\rm Ker}\{h_{2p-2}:\pi_{2p-2}(\Omega
P^3(p))\to H_{2p-2}(\Omega P^3(p))\}.
$$
Taking $\mathcal G$ to be a simplicial model for $\Omega P^3(p)$
with $\mathcal G_3$ generated by degeneracies, we obtain the
following example. Set $G=\mathcal G_{2p-2},\ K_i={\rm Ker} d_i:
\mathcal G_{2p-2}\to \mathcal G_{2p-3},$ then, by proposition
\ref{simple}, the generalized dimension quotient
$$
\frac{D(G,\|{\bf k}_0,\dots, {\bf k}_{2p-2}\|)}{\|K_0,\dots,
K_{2p-2}\|}
$$
contains a subgroup $\mathbb Z/p$.

In the case $p=2$, we can choose $$\mathcal G_1=F(\sigma),\
\mathcal G_2=F(a,b,c)$$ with the face maps
$$
d_0:\begin{cases} a\mapsto \sigma^2\\ b\mapsto \sigma\\
c\mapsto 1\end{cases},\ d_1:\begin{cases} a\mapsto 1\\
b\mapsto \sigma\\ c\mapsto \sigma\end{cases},\ d_2: \begin{cases}
a\mapsto 1\\ b\mapsto 1\\ c\mapsto \sigma\end{cases}
$$
Taking the kernels $R={\rm Ker}(d_0),\ S={\rm Ker}(d_1),\ T={\rm
Ker}(d_2)$, we obtain an example discussed in introduction, with
$$
\frac{D(G, \|{\bf r},{\bf s},{\bf t}\|)}{\|R,S,T\|}\simeq \mathbb
Z/2.
$$

\noindent{\bf Conjecture.} For a prime $p$, any group $G$ and its
normal subgroups $R_1,\dots, R_n,\ n\geq 2$, the quotient
$$
\frac{D(G,\|{\bf r}_1,\dots, {\bf r}_n\|)}{\|R_1,\dots, R_n\|}
$$
is $p$-torsion free provided $n<2p-1$.

\subsection{Connectivity conditions} Let G be a group with normal subgroups $R_1,\dots, R_n,$ $n\geq 2$. Recall the connectivity
condition from \cite{EM}. The $n$-tuple of normal subgroups
$(R_1,\dots, R_n)$ is called {\it connected} if either $n\geq 2$
or $n\geq 3$ and for all subsets $I,J\subseteq \{1,\dots, n\},$
with $|I|\geq 2,$ $|J|\geq 1$, the following holds
$$
\left(\bigcap_{i\in I} R_i\right)\left(\prod_{j\in
J}R_j\right)=\bigcap_{i\in I}\left(R_i(\prod_{j\in J}R_j)\right)
$$
Now consider the homotopy colimit $X$ of classifying spaces
$B(G/\prod_{i\in I}R_i)$, where $I$ ranges over all proper subsets
$I\subset \{1,\dots, n\}$. Then, if for any $i=1,\dots, n$, the
$n-1$-tuple of normal subgroups $(R_1,\dots, \hat R_i,\dots, R_n)$
is connected, then
$$
\pi_n(X)=\frac{R_1\cap \dots \cap R_n}{\|R_1,\dots, R_n\|}.
$$
This is proved in \cite{EM}. The proof of theorem 1 from \cite{EM}
together with results of section \ref{section3}, namely, together
with the construction of the functor $\Pi$, imply the following
\begin{proposition}\label{p12}
If for any $i=1,\dots, n$, the $n-1$-tuple of normal subgroups
$(R_1,\dots, \hat R_i,\dots, R_n)$ is connected, then there is a
commutative diagram
$$
\xyma{\frac{R_1\cap \dots \cap R_n}{\|R_1,\dots,
R_n\|}\ar@{->}[r]\ar@{=}[d] & \frac{{\bf r}_1\cap \dots \cap {\bf
r}_n}{\|{\bf r}_1,\dots, {\bf r}_n\|}\ar@{->}[d]\\
\pi_{n-1}(\Omega X)\ar@{->}[r]^{h_{n-1}\Omega} & H_{n-1}(\Omega
X)}.
$$
\end{proposition}
One can use proposition \ref{p12} for proving the above conjecture
about the $p$-torsion in the generalized dimension quotient in
some particular cases.

\end{document}